\documentclass[12pt,reqno]{amsart}

\usepackage{fixltx2e} 
\usepackage[utf8]{inputenc} 
\usepackage{amssymb, latexsym, stmaryrd, amsthm, dsfont, amsfonts, amsbsy, amsmath, mathrsfs} 
\usepackage{mathtools} 
\usepackage{bm, bbm} 
\usepackage{enumerate} 
\usepackage{verbatim} 
\usepackage{url}   
\usepackage{lscape} 
\usepackage{centernot}
\usepackage[dvipsnames]{xcolor}
\usepackage[colorinlistoftodos]{todonotes} 
\usepackage{nicefrac} 
\usetikzlibrary{calc}
\usepackage{todonotes}

\usepackage{microtype,tikz,tikz-cd}  
\makeatletter 
\def\MT@register@subst@font{\MT@exp@one@n\MT@in@clist\font@name\MT@font@list
 \ifMT@inlist@\else\xdef\MT@font@list{\MT@font@list\font@name,}\fi}
\makeatother

\makeatletter 
\newcommand{\myitem}[1]{%
\item[(#1)]\protected@edef\@currentlabel{#1}%
}
\makeatother

\usepackage[pdftex,bookmarks,bookmarksnumbered,linktocpage,   
         colorlinks,linkcolor=blue,citecolor=blue]{hyperref}
         
\usepackage[capitalize,noabbrev]{cleveref} 

\newcommand{\bit}{\begin{itemize}}    
\newcommand{\eit}{\end{itemize}}
\newcommand{\ben}{\begin{enumerate}}
\newcommand{\een}{\end{enumerate}}

\newcommand{\bde}{\begin{description}}
\newcommand{\ede}{\end{description}}

\theoremstyle{theorem}
\newtheorem{Theorem}{Theorem}[section]
\newtheorem{Theorem-n}{Theorem}

\newtheorem{Proposition}[Theorem]{Proposition}
\newtheorem{Modal Sahlqvist Theorem}[Theorem]{Modal Sahlqvist Theorem}
\newtheorem{Intuitionistic Sahlqvist Theorem}[Theorem]{Intuitionistic  Sahlqvist Theorem}
\newtheorem{Esakia Duality}[Theorem]{Esakia Duality}

\newtheorem{Main Lemma}[Theorem]{Main Lemma}
\newtheorem{Compactness Theorem}[Theorem]{Compactness Theorem}
\newtheorem{Los Theorem}[Theorem]{\LL o\'s' Theorem}
\newtheorem{Isbell Theorem}[Theorem]{Isbell's Zigzag Theorem}
\newtheorem{Diagram Lemma}[Theorem]{Diagram Lemma}

\newtheorem{Transfer Lemma}[Theorem]{Transfer Lemma}
\newtheorem{Subdirect Decomposition Theorem}[Theorem]{Subdirect Decomposition Theorem}

\newtheorem{Corollary}[Theorem]{Corollary}
\newtheorem{Claim}[Theorem]{Claim}

\theoremstyle{definition}

\theoremstyle{remark}

\newtheorem{Remark}[Theorem]{Remark}

\let\leq=\leqslant
\let\nleq=\nleqslant

\usepackage[mathscr]{euscript}
 \let\mathscr\relax 
\usepackage[scr]{rsfso}

\renewcommand{\int}{\mathsf{int}\,}

\bmdefine{\A}{A} 
\bmdefine{\C}{C}                                
                              
\bmdefine{\2}{2}
\bmdefine{\B}{B}
\bmdefine{\D}{D}
\bmdefine{\E}{E}
\bmdefine{\Luk}{L}
\bmdefine{\Term}{T} 
\bmdefine{\Free}{F}
\bmdefine{\Fb}{F}

\usepackage{scalerel}

\newcommand{\up}{{\uparrow}}

\let\LL\L 
\renewcommand{\L}{\mathscr{L}}

\DeclareUnicodeCharacter{001B}{}

\begin{document}

\title{Bounded depth in Hilbert algebras}

\author{Luca Carai, Miriam Kurtzhals, and Tommaso Moraschini}



\begin{abstract}
  \emph{Hilbert algebras} are the implicative subreducts of Heyting algebras. It is shown that having depth $\leq n$ is an equational condition in Hilbert algebras. This generalizes an analogous well-known result in the setting of Heyting algebras. 
\end{abstract}
\date{\today}

\maketitle

\section{Hilbert algebras and implicative filters}

An algebra $\langle A; \to \rangle$ is
said to be a \emph{Hilbert algebra} 
when it is an implicative subreduct of 
a Heyting algebra (see, e.g., \cite{Di65,He50,Ho62}). In every Hilbert 
algebra $\A$ we have $a \to a = b \to b$ for all $a, b \in A$. Therefore, we
can define a constant $1 = a \to a$ for $a \in A$. In addition, the binary relation $\leq$ on $A$ defined by the rule
\[
a \leq b \iff a \to b = 1
\]
is a partial order on $A$ with maximum $1$.

Let $\A$ be a Hilbert algebra. A set $F \subseteq A$ is an \emph{implicative filter} of $\A$ when $1 \in F$ and for 
all $a, b \in A$,
\[
\text{if }(a \in F \text{ and }a \to b \in F)\text{, then }b \in F.
\]
When ordered under inclusion, the set of implicative filters of $\A$ forms a distributive lattice that we denote by $\mathsf{Fi}(\A)$ (see, e.g., \cite[Prop.\ 2.9]{CeJa}). Every implicative filter of $\A$ is an upset (see, e.g., \cite[Lem.\ 2.8(1)]{CeJa}). Moreover, for every $a \in A$ the principal upset 
\[
{\uparrow}a = \{ b \in A : a \leq b \}
\]
is an implicative filter of $\A$ (see, e.g., \cite[p.\ 192]{CeJa}).  We denote by $\mathsf{Fg}^\A(-)$ the closure operator of implicative filter generation on $\A$ (see \cite[Lem.\ 2.3]{MR811800}).

\begin{Proposition}\label{Prop : generation}
Let $\A$ be a Hilbert algebra and $X \subseteq A$. Then
\[
\mathsf{Fg}^\A(X) = \{ a \in A : a = 1 \text{ or }b_1 \to ( b_2 \to ( \dots (b_n \to a)\dots )) = 1 \text{ for some }b_1, \dots, b_n \in X\}.
\]
\end{Proposition}

    From Proposition \ref{Prop : generation} and the fact that Hilbert algebras are subreducts of Heyting algebras we deduce the following.

\begin{Corollary}\label{Cor : generation}
Let $\A$ be a Hilbert algebra and $X \cup \{ c \} \subseteq A$.  Then
\begin{align*}
     \mathsf{Fg}^\A(X \cup \{ c \}) & = \{ a \in A : a = 1 \text{ or }b_1 \to ( b_2 \to ( \dots (b_n \to (c \to a))\dots )) = 1\\
     & \, \, \, \, \, \, \, \, \text{ for some }b_1, \dots, b_n \in X\}.
\end{align*}
\end{Corollary}

Let $\A$ be a Hilbert algebra. With every $F \in \mathsf{Fi}(\A)$ we associate a congruence of $\A$ as follows:
\[
\theta_F = \{ \langle a, b \rangle \in A \times A : a \to b, b \to a \in F \}.
\]
The quotient algebra $\A / \theta_F$ is a Hilbert algebra that we denote by $\A / F$. Similarly, we denote the equivalence class of an element $a$ of $A$ under $\theta_F$ by $a / F$. The map $F \longmapsto \theta_F$ is a lattice isomorphism from $\mathsf{Fi}(\A)$ to the congruence lattice of $\A$. Because of this, the following is a consequence of the Correspondence Theorem (see, e.g., \cite[Thm.\ II.6.20]{BuSa00}).

\begin{Theorem}\label{Thm : correspondence}
    Let $\A$ be a Hilbert algebra and $F \in \mathsf{Fi}(\A)$. Then there exists an isomorphism $h \colon {\uparrow}F \to \mathsf{Fi}(\A / F)$, where ${\uparrow}F$ is the subposet of $\mathsf{Fi}(\A)$ with universe $\{ G \in \mathsf{Fi}(\A) : F \subseteq G \}$.
\end{Theorem}

We recall that an element of lattice $\A$ is said to be \emph{meet irreducible} when it is neither the maximum of $\A$ nor it can be written as the meet of two larger elements. Given a Hilbert algebra $\A$, we denote the poset of meet irreducible elements of $\mathsf{Fi}(\A)$ by $\A_*$. Owing to the distributivity of $\mathsf{Fi}(\A)$, the poset  $\A_*$ coincides with that of  meet prime elements of $\mathsf{Fi}(\A)$ (see, e.g., \cite[p.\ 192]{CeJa}). We will make use of the following observation (see, e.g., \cite[Prop.\ 2.10(1)]{CeJa}).

\begin{Proposition}\label{Prop : separation}
    Let $\A$ be a Hilbert algebra, $a \in A$, and $F \in \mathsf{Fi}(\A)$. If $a \notin F$, there exists $G \in \A_*$ such that $F \subseteq G$ and $a \notin G$.
\end{Proposition}

Since principal upsets of Hilbert algebras are always implicative filters, the following is an immediate consequence of Proposition \ref{Prop : separation}.

\begin{Corollary}\label{Cor : separation}
  Let $\A$ be a Hilbert algebra, $a \in A$, and $a, b \in A$.   If $a \nleq b$, there exists $F \in \A_*$ such that $a \in F$ and $b \notin F$.
\end{Corollary}

\begin{Remark}\label{Rem : filters}
An algebra $\langle A; \land, \to \rangle$ is
said to be an \emph{implicative semilattice} 
when it is a $\langle \land, \to \rangle$-subreduct of 
a Heyting algebra (see, e.g., \cite{Kh81}). Let $\A$ be an implicative semilattice (resp.\ a Heyting algebra) and $\A_\to$ its Hilbert algebra reduct. Then $\mathsf{Fi}(\A_\to)$ coincides with the lattice of filters of $\A$ and, consequently, $(\A_\to)_*$ coincides with the poset of meet irreducible filters of $\A$. We recall that the meet irreducible filters of a Heyting algebra are precisely its prime filters. Therefore, in the case where $\A$ is a Heyting algebra,  $(\A_\to)_*$ is the poset of prime filters of $\A$.
\qed
\end{Remark}

\section{The main result}

We say that a Hilbert algebra $\A$ has \emph{depth} $\leq n$ for $n \in \mathbb{N}$ when $\A_*$ does not contain any chain of length $n+1$. We will show that this condition can be described equationally.\footnote{For Heyting algebras, this is well known (see, e.g., \cite[p.\ 43]{ChZa97}).} To this end, for every $n \in \mathbb{N}$ we define recursively a formula $d_n(x_0, \dots, x_n)$ as follows:
    \begin{align*}
        d_0(x_0) &= x_0;\\
    d_{n+1}(x_0, \dots, x_{n+1}) &= ((x_{n+1} \to d_n(x_0, \dots, x_n)) \to x_{n+1}) \to x_{n+1}.
    \end{align*}

    \begin{Theorem}\label{Thm : depth}
    Let $\A$ be a Hilbert algebra and $n \in \mathbb{N}$. Then $\A$ has depth $\leq n$ iff  $\A \vDash d_n \approx 1$. 
\end{Theorem}

\begin{proof}
We begin by proving the implication from left to right. To this end, we reason by contraposition. We will prove by induction on $n$ that $\A \nvDash d_n \thickapprox 1$ implies that $\A$ is not of depth $\leq n$ for every Hilbert algebra $\A$.

In the base case, $n = 0$. Consider a Hilbert algebra such that $\A \nvDash d_0 \thickapprox 1$. Then there exists $a_0 \in A$ such that $a_0 = d_0(a) < 1$. By Corollary \ref{Cor : separation} there exists $F \in \A_*$ such that $1 \in F$ and $a_0 \notin F$. Consequently, $\A_*$ is nonempty. It follows that $\A$ is not of depth $\leq 0$, as desired. 

For the inductive step, consider a Hilbert algebra $\A$ such that $\A\nvDash d_n \thickapprox 1$. Moreover, the inductive hypothesis ensures that Hilbert algebras $\B$ such that $\B \nvDash d_{n-1} \approx 1$ are not of depth $\leq n-1$. Since $\A\nvDash d_n \thickapprox 1$,  there exist $a_0, \dots, a_n \in A$ such that $d_n(a_0, \dots, a_n)  < 1$. 
By the definition of $d_n$ this implies 
\begin{equation}\label{Eq : first step 1}
    ((a_n \to d_{n-1}(a_0, \dots, a_{n-1})) \to a_n) \to a_n < 1.
\end{equation}
To improve readability, from now on we write
\[
b = d_{n-1}(a_0, \dots, a_{n-1}).
\]

From (\ref{Eq : first step 1}) it follows that 
$(a_n \to b) \to a_n \nleq a_n$. By Corollary \ref{Cor : separation} there exists $F_0 \in \A_*$ such that $(a_n \to b) \to a_n \in F_0$ and $a_n \notin F_0$. Define $F = \mathsf{Fg}^\A(F_0  \cup \{a_n\})$.

\begin{Claim}\label{Claim : n notin F}
We have $b \notin F$.
\end{Claim}

\begin{proof}[Proof of the Claim]
    Suppose the contrary, with a view to contradiction. Observe that $b \ne 1$. For if $b =1$, from $(a_n \to b) \to a_n \nleq a_n$ and $(a_n \to 1) \to a_n = a_n$ it would follow $a_n \nleq a_n$, which is false. Hence, $b \ne 1$, as desired. Therefore, from Corollary \ref{Cor : generation} it follows that  $c_1 \to ( c_2 \to ( \dots (c_m \to (a_n \to b))\dots )) = 1$  for some $c_1, \dots, c_m \in F_0$. Since $c_1, \dots, c_m \in F_0$ and $c_1 \to ( c_2 \to ( \dots (c_m \to (a_n \to b))\dots )) = 1 \in F_0$, the assumption that $F_0$ is an implicative filter yields $a_n \to b \in F_0$. Together with the assumption that $(a_n \to b) \to a_n \in F_0$, this yields $a_n \in F_0$, a contradiction.  
\end{proof}

Recall that $b = d_{n-1}(a_0, \dots, a_{n-1})$ by definition. Together with Claim \ref{Claim : n notin F}, this yields 
\[
d_{n-1}(a_0/F, \dots, a_{n-1}/F) = d_{n-1}(a_0, \dots, a_{n-1})/F < 1/F,
\]
whence $\A / F \nvDash d_{n-1} \thickapprox 1$. Hence, by the inductive hypothesis $\A/F$ is not of depth $\leq n-1$. Consequently, there exists a chain $G_1 \subsetneq \dots \subsetneq G_n$ in $(\A / F)_*$. 
By  Theorem \ref{Thm : correspondence} this yields a chain $F_1 \subsetneq \dots \subsetneq F_n$ in $\A_*$, with $F \subseteq F_1$.  
Moreover, since $F_0 \subseteq F \subseteq F_1$ and $a_n \in F - F_0 \subseteq F_1 - F_0$, we conclude that $F_0 \subsetneq F_1 \subsetneq F_2 \subsetneq \dots \subsetneq F_n$ is a chain in $\A_*$ witnessing that $\A$ is not of depth $\leq n$, as desired.

Next, we prove the implication from right to left. To this end, we begin by showing by induction on $n$ that if there exists a chain $F_0 \subsetneq \dots \subsetneq F_n$ in $\A_*$, then there exists a chain $a_0 < \dots < a_n < 1$ in $\A$ such that $\{a_0, \dots, a_n,1\}$ is a subuniverse of $\A$ and $a_n \notin F_0$. 
In the base case, $n=0$ and there exists $F_0 \in \A_*$. Then $F_0 \neq A$. Therefore, there exists $a_0 \in A-F_0$.
It follows that $a_0 < 1$ and $\{a_0,1\}$ is a subuniverse of $\A$. For the inductive step, consider a chain $F_0 \subsetneq \dots \subsetneq F_n$ in $\A_*$ with $n > 0$.  Then $F_1 \subsetneq \dots \subsetneq F_n$ is also a chain in $\A_*$. By the inductive hypothesis there exists a chain $a_0 < \dots < a_{n-1} < 1$ in $\A$ such that $\{a_0, \dots, a_{n-1},1\}$ is a subuniverse of $\A$ and $a_{n-1} \notin F_1$. 
Let 
\[
G = \{ b \in A : a_{n-1} \leq b \text{ and } b \to a_{n-1} = a_{n-1}\}.
\]
It is easy to check that $G  \in \mathsf{Fi}(\A)$.

\begin{Claim}
$F_1 \cap G \nsubseteq F_0$.
\end{Claim}

\begin{proof}[Proof of the claim.]
Suppose the contrary, with a view to contradiction. As $F_0$ is meet irreducible in $\mathsf{Fi}(\A)$, it is also meet prime. Therefore, from $F_1 \not \subseteq F_0$ and $F_1 \cap G \subseteq F_0$ it follows that  $G \subseteq F_0$.

Recall that $n > 0$. Then $F_0 \subsetneq F_1$ by assumption. Therefore, there exists $b \in F_1 - F_0$. We will show that $\up b \cap \up (b \to a_{n-1}) \subseteq G$. 
Consider $c \in \up b \cap \up (b \to a_{n-1})$. Then $b \leq c$ and $b \to a_{n-1} \leq c$. It follows that $c \to a_{n-1} \leq b \to a_{n-1}$ and $c \to a_{n-1} \leq (b \to a_{n-1}) \to a_{n-1}$. Therefore, $c \to a_{n-1} \leq a_{n-1}$. As $a_{n-1} \leq c \to a_{n-1}$, we obtain $c \to a_{n-1} = a_{n-1}$. We also have $a_{n-1} \leq b \to a_{n-1} \leq c$. As $c \to a_{n-1} = a_{n-1}$ and $a_{n-1} \leq c$, the definition of $G$ ensures that $c \in G$, as desired.

Since $G \subseteq F_0$ by assumption and $b \cap \up (b \to a_{n-1}) \subseteq G$, we have $\up b \cap \up (b \to a_{n-1}) \subseteq F_0$. 
As $F_0 \in \A_*$, we know that $F_0$ is meet prime. Consequently,  either $\up b \subseteq F_0$ or $\up (b \to a_{n-1}) \subseteq F_0$, and so either $b \in F_0$ or $b \to a_{n-1} \in F_0$. That $b \in F_0$ is impossible because $b \in F_1-F_0$. That $b \to a_{n-1} \in F_0$ is impossible as well. Indeed, $b \to a_{n-1} \in F_0 \subseteq F_1$ and $b \in F_1$ imply $a_{n-1} \in F_1$, which contradicts the inductive hypothesis.     
\end{proof}

The claim above implies that there exists $a_n \in (F_1 \cap G)-F_0$. As $a_n \notin F_0$, we have $a_n <1$. From $a_n \in G$ it follows that $a_{n-1} \leq a_n$ and $a_n \to a_{n-1} = a_{n-1}$. Then $a_{n-1} < a_n$ because otherwise $a_{n-1} = a_n \to a_{n-1} = 1$. It only remains to show that $\{a_0, \dots, a_n,1\}$ is a subuniverse of $\A$. To this end, it suffices to prove that $a_n \to a_i = a_i$ for every $i \leq n-1$. By what we have observed above, $a_n \to a_{n-1} = a_{n-1}$, so we can assume $i < n-1$. 
Since $a_{n-1} \leq a_n$, it follows that $a_n \to a_i \leq a_{n-1} \to a_i = a_i$, where the last equality holds because the chain $a_0 < \dots < a_{n-1} < 1$ forms a subuniverse of $\A$ and this forces $a_{n-1} \to a_i = a_i$. 
As $a_i \leq a_n \to a_i$, we obtain $a_n \to a_i = a_i$ for every $i < n-1$. This concludes the proof by induction. 

To finish the proof, we need to show that if $\A$ is not of depth $\leq n$, then $\A \nvDash d_n \approx 1$. Assume that $\A_*$ contains a chain of size $n+1$. By what we have shown above there is a subalgebra $\A$ of the form $a_0 < \dots < a_n < 1$. 
Moreover, this subalgebra does not validate $d_n \approx 1$ because $d_i(a_0, \dots, a_i) = a_i$ for every $i \leq n$. Hence, we conclude that $\A \nvDash d_n \thickapprox 1$.
\end{proof}

Let $\A$ be an implicative semilattice (resp.\ a Heyting algebra). We say that $\A$ has \emph{depth} $\leq n$ for $n \in \mathbb{N}$ when the poset of  meet irreducible filters (resp.\ prime filters) of $\A$ does not contain any chain of length $n+1$. 

\begin{Corollary}
       Let $\A$ be an implicative semilattice or a Heyting algebra and $n \in \mathbb{N}$. Then $\A$ has depth $\leq n$ iff  $\A \vDash d_n \approx 1$. 
\end{Corollary}

\begin{proof}
    Immediate from Remark \ref{Rem : filters} and Theorem \ref{Thm : depth}.
\end{proof}

\bibliographystyle{plain}

\end{document}